\documentclass[11pt, reqno]{amsart}
\usepackage[thm]{macros_proof}
\newgeometry{margin=1in}
\usepackage{natbib}

\newcommand\numberthis{\addtocounter{equation}{1}\tag{\theequation}}
\title{On the robustness of posterior means}
\author{Jiafeng Chen}
\date{\today. }

\begin{document}

\maketitle

Consider an observation $X$ whose likelihood is $X \mid \theta \sim \Norm(\theta,
\sigma^2)$ for some known
$\sigma^2$. There are two priors for $\theta$, denoted by $G_0$ and $G_1$. Suppose both
priors have mean zero and have finite variances bounded by $V > 0$. Consider the decision
problem of estimating $\theta$ under squared error, with $L(a,\theta) = (a-\theta)^2$. For
the Bayesian with prior $G_1$, the Bayes decision rule is the posterior mean $\E_{G_1}
[\theta \mid X]$ under the prior $G_1$. This decision rule attains Bayes risk under the
prior $G_0$\[
R(G_1, \sigma; G_0) \equiv \E_{\theta \sim G_0}\bk{
	(\E_{G_1}[\theta \mid X] - \theta)^2
}. 
\] We can think of $R(G_1, \sigma; G_0)$ as a measure of decision quality under
 disagreement. It measures the quality of $G_1$'s decision from $G_0$'s point of view.
 When $G_1 \neq G_0$, how large can $R (G_1,
\sigma; G_0)$ be?

This note bounds $R(G_1, \sigma; G_0)$ uniformly over $G_1, G_0, \sigma$.\footnote
{The previous version of this note (arXiv:2303.08653v1) claimed that $R
(G_1, \sigma; G_0)$ is uniformly bounded over all $G_0, G_1, \sigma>0$, subjected to the
constraints on the moments of $G_0,G_1$. Regrettably, it
contained a critical error that rendered its proof incorrect. In particular, in that
version, the display before (A1) on p.4 is incorrect: Posterior means of mixture priors
are mixtures of posterior means under each mixing component, but the mixing weights are
posterior probabilities
assigned to each mixing component; thus, the mixing weights depend on the data rather
than being fixed.

This version partially restores the main result in the previous version. \Cref{thm:main}
shows that the maximum Bayes risk under $G_0$ is uniformly bounded over all $G_0, G_1,
\sigma^2$ where $G_1$ satisfies an additional tail condition \eqref{eq:g1_cond}. The
bound we obtain depends on the tail condition, and thus \cref{thm:main} is insufficient
for the result in the previous draft.
}
Since \[ R(G_1, \sigma; G_0) \le 2\pr{
\E_{G_0}[\E_{G_1}[\theta \mid X]^2] + \E_{G_0}[\theta^2] } \le 2V + 2\E_{G_0}[\E_{G_1}
 [\theta \mid X]^2],\] it thus suffices to bound $\E_{G_0}[\E_{G_1}[\theta \mid X]^2]$
 modulo constants. That is,  it suffices to bound the 2-norm $\norm{\E_{G_1}
 [\theta \mid X]}^2$ under the law $X \sim \Norm(0,\sigma^2)
\star G_0$. 

The rest of the note shows that this quantity is uniformly bounded over all $G_0,
G_1, \sigma^2 > 0$. Specifically, \cref{lemma:bound_small_sigma} shows that for all $G_0,
G_1$ that are mean zero and have variance bounded by $V$, $\E_{G_0}[\E_{G_1}[\theta
\mid X]^2]$ is bounded by a constant that depends only on $(V, \sigma^2)$. This bound is
 large when $\sigma^2$ is large. To improve this bound, \cref{thm:main} then shows that,
 if $G_1$ additionally satisfies some conditions on its tail behavior, $\E_ {G_0}[\E_
 {G_1}[\theta
\mid X]^2]$ is bounded by a constant that depends only on $V$ and the tail condition---and
does not depend on $\sigma$.

\begin{lemma}
\label{lemma:bound_small_sigma}
Suppose $G_0, G_1$ have mean zero and variances bounded by $V$, then
\[
\E_{G_0}[\E_{G_1}[\theta \mid X]^2] \le 6V + 4\sigma^2
\]
uniformly over $G_0, G_1, \sigma^2$. 
\end{lemma}
\begin{proof}
Let $f_{G,\sigma}(x) = \int f_X(x \mid \theta) \,G(d\theta)$.
\citet{jiang2020general} (Lemma 1) shows that \[
\pr{\frac{f'_{G,\sigma}(x)}{f_{G,\sigma}(x)}}^2 \le \frac{1}{\sigma^2} \log \pr{\frac{1}
{2\pi \sigma^2 f^2_{G,\sigma}(x)}}.
\]
Plugging in the bound \eqref{eq:jensen_denom} in \cref{lemma:jensen_denom}, we have that
pointwise in $X$, \[
\pr{\E_{G_1}[\theta \mid X] - X}^2 = \pr{\sigma^2\frac{f'_{G_1,\sigma}(x)}{f_
{G_1,\sigma} (x)}}^2 \le \sigma^4 \frac{1}
{\sigma^2} \frac{X^2 + V}
{\sigma^2} = X^2 + V,
\]
where the first equality is due to Tweedie's formula. 

Now, note that \[
(\E_{G_1}[\theta \mid X])^2 \le 2 \pr{
	(\E_{G_1}[\theta \mid X] - X)^2 + X^2
} \tag{$(a+b)^2 \le 2(a^2 + b^2)$}.
\]
Hence, \[
\E_{G_0}\bk{
	\E_{G_1}[\theta \mid X])^2
} \le 2 \E_{G_0}[2X^2+V] \le 2(2(V+\sigma^2) + V) = 6V + 4\sigma^2. \qedhere
\]
\end{proof}

To show a more powerful bound, we require a stronger condition on the tails of $G_1$ and
derive bounds that are independent of $\sigma$ but are dependent on the tail conditions.
In particular, assume
\[ \max\pr{1-G_1(s), G_1(-s) }\le C_{G_1} s^{-k} \numberthis\label{eq:g1_cond}
\]
for some $k > 2$ and $C_{G_1} > 0$, for all $s > 0$. We will also assume that $\E
_{G_1}[\theta^2 \mid X]$ exists almost surely. Note that if $\E_{G_1} |\theta|^
{2+\epsilon} < m$, then $k$ can be taken to be $2+\epsilon$ and $C_{G_1}$ can be taken to
be $m$ by Markov's inequality. In the rest of the proof, we let $C_t < \infty$ denote a
positive constant that depends only on $t$. Occurrences of
$C_t$ might correspond to different constant values.

\begin{theorem}
\label{thm:main}
Suppose $k > 2.$ There exists a constant $Q < \infty$ that depends solely on $(C_ {G_1},
k, V)$ such that, uniformly for all $(G_0, G_1)$ and $\sigma \in \R$, where (i) $G_0,G_1$
have mean zero and variance bounded by $V$ and (ii) $G_1$ satisfies
\eqref{eq:g1_cond} with $(C_{G_1},
k)$ where $\E_{G_1}[\theta^2 \mid X]$ exists for almost every $X$, \[
\E_{G_0}\bk{\E_{G_1}[\theta\mid X]^2} \le Q.
\]
\end{theorem}

\begin{proof}
Assume that $ \sigma^2 \ge 1 $. For all $\sigma^2 < 1$, we can apply 
\cref{lemma:bound_small_sigma}
so that $\E_{G_0}\bk{\E_{G_1}[\theta\mid X]^2} \le 6V + 4$. 

Observe that \begin{align*}
\E_{G_0}\bk{\E_{G_1} [\theta \mid X]^2} &\le \E_{G_0} \bk{\E_{G_1}[\theta^2 \mid X]}  \tag{Jensen's
inequality} \\
&= \E_{G_0}\bk{\int_{0}^\infty P_{G_1}(\theta^2 > t \mid X) \,dt} \\
&= 2\E_{G_0}\bk{\int_0^\infty s P_{G_1}(|\theta| > s \mid X)\,ds \tag{Change of variable $s
=
\sqrt{t}$}}
\\
&= 2\E_{G_0}\bk{\int_0^\infty s P_{G_1}(\theta > s \mid X)\,ds + \int_0^\infty s P_{G_1}
(-\theta > -s
\mid X)\,ds}.
\end{align*}
Therefore, it suffices to bound the first term, since the second term follows by a
symmetric argument. We do so in the remainder of the proof.

Writing out the first term as an integral:
\begin{align*} &\E_{G_0}
\bk{\int_0^\infty s P_ {G_1} (\theta > s \mid
X)\,ds} \\ &= \int_ {\mu=-\infty}^\infty\int_{x=-\infty}^\infty \int_{s=0}^\infty  s
P_{G_1}[\theta > s
\mid
X=x]\, ds f_X (x\mid \mu)
\,dx\,G_0(d\mu)\\
&= \int_
{\mu=-\infty}^\infty  \int_{s=0}^\infty s \int_{x=-\infty}^\infty P_{G_1}[\theta > s
\mid
X=x]  f_X (x\mid \mu) \,dx \,ds \,G_0(d\mu). \tag{Fubini's theorem}
\end{align*}

The outer integral in $\mu$ can be decomposed into $|\mu| \le \sigma$ and $|\mu| >
\sigma$: \begin{align}
&\E_{G_0}\bk{\int_0^\infty s P_ {G_1} (\theta > s \mid
X)\,ds} \nonumber\\
&= \int_
{|\mu| > \sigma}  \int_{s=0}^\infty s \int_{x=-\infty}^\infty P_{G_1}[\theta > s
\mid
X=x]  f_X (x\mid \mu) \,dx \,ds \,G_0(d\mu)\label{eq:small_sigma}
\\
&\quad+ \int_
{|\mu| < \sigma} \int_{s=0}^\infty s \int_{x=-\infty}^\infty P_{G_1}[\theta > s
\mid
X=x]  f_X (x\mid \mu) \,dx \,ds \,G_0(d\mu) \label{eq:large_sigma}
\end{align}

First, we consider \eqref{eq:small_sigma}. Decompose the integral in $x$ further along $x
\le s/2$ and $x > s/2$: 
\begin{align}
\eqref{eq:small_sigma} &= \int_{|\mu| > \sigma} \int_0^\infty s \int_{s/2}^{\infty}P_{G_1}
(\theta > s
\mid X=x) f_X(x\mid \mu) \,dx \,ds \,G_0(d\mu)\label{eq:term3}
\\&\quad + \int_{|\mu| > \sigma} \int_0^\infty s \int_{-\infty}^{s/2} P_{G_1}
(\theta > s
\mid X=x) f_X(x\mid \mu) \,dx \,ds\,G_0(d\mu) \label{eq:term2}
\end{align}

For large $\mu$ and large $x$ \eqref{eq:term3}, we have that
\begin{align*}
\eqref{eq:term3} &\le \int_{|\mu| > \sigma} \int_0^\infty s \int_{s/2}^{\infty} f_X(x\mid
\mu) \,dx ds G_0(d\mu) \tag{$P_{G_1}
(\theta > s
\mid X=x) \le 1$}\\
&= \int_{|\mu| > \sigma} \int_0^\infty s P(X > s/2 \mid \mu)\,ds \,G_0(d\mu) \\ 
&\le C\int_{|\mu| > \sigma} \underbrace{\E[X^2 \mid \mu]}_{\mu^2 + \sigma^2 \le
2\mu^2}\, G_0(d\mu)  \le C \int
2\mu^2
\,G_0(d\mu) \le C_{V}.\tag{$\int 2sP(X > s \mid \mu)\,ds \le \E[X^2 \mid \mu]$}
\end{align*}

For large $\mu$ and small $x$ \eqref{eq:term2}, note that for $x < s/2< s$, by 
\cref{lemma:jensen_denom}
\begin{align*}
P_{G_1}(\theta > s \mid X=x) &\le C_{V} e^{x^2/(2\sigma^2)} e^{-\frac{1}{2\sigma^2}(x-s)^2}
(1-G_1(s)). \tag{$f_X(x\mid \theta) \le \frac{1}{\sqrt{2\pi} \sigma} e^{-\frac{1}
{2\sigma^2} (x-s)^2}$}
\end{align*}
Now, integrating the above display with respect to $f_X(x\mid \mu)\,dx$ yields \[
 \int_{-\infty}^{s/2}P_{G_1}
(\theta > s
\mid X=x) f_X(x\mid \mu) \,dx \le C_{V}(1-G_1(s)) \cdot \frac{\sigma^2}{s} \le 
C_{V}(1-G_1(s)) \frac{\mu^2} {s} \tag{$|\mu| > \sigma$ for \eqref{eq:small_sigma}}
\]
Finally, integrating it again with respect to $s$ yields \[
\int_0^\infty s \times C_{V}(1-G_1(s)) \frac{\mu^2} {s} \,ds = C_V \mu^2 \int_{0}^\infty
1-G_1(s)\,ds \le C_V \mu^2 \E_{G_1}[|\theta|] \le C_V \mu^2.
\]
Therefore, \[
\eqref{eq:term2} \le  C_V \E_{G_0}\mu^2 \le C_V.
\]
This shows that \eqref{eq:small_sigma} is uniformly bounded.

Moving on to \eqref{eq:large_sigma}, we first decompose the integral on $s$ into $s \le K$
and $s > K$, for some $K \ge e$ to be chosen: \[
\eqref{eq:large_sigma} \le \underbrace{\int_{|\mu| < \sigma} \int_0^K s \,ds \,G_0(d\mu)}_
{\le K^{2}/2} + \int_
{|\mu| < \sigma} \int_K^\infty s \int_{-\infty}^\infty P_{G_1}(\theta > s \mid X=x) f_X
(x\mid \mu) \,dx\,G_0(d\mu) \numberthis \label{eq:large_sigma_bound}
\] Thus we focus on the piece where $s > K$. Fix \[u = C \sigma \sqrt{\log (s)}\] for some
 $C
\ge 2$ to be chosen. On $s > K$,
$u/\sigma > 2$ and thus $\frac{u}{\sigma} - 1 > \frac{u}{2\sigma}$.
Observe that by \cref{lemma:jensen_denom} and the fact that $\sigma > 1$, \[
P_{G_1}(\theta > s \mid X=x) \le C_V \exp\pr{\frac{x^2}{2\sigma^2}} (1-G(s)). \numberthis
\label{eq:survival_bound}
\]
Therefore, \begin{align*}
&\int_{-\infty}^\infty P_{G_1}(\theta > s \mid X=x) f_X
(x\mid \mu)\,dx \\
&\le \int_{|x| \le u} P_{G_1}(\theta > s \mid X=x) f_X
(x\mid \mu)\,dx + \P(|X|>u \mid \mu) \tag{$P_{G_1}(\theta > s \mid X=x) \le 1$}\\
&\le C_V e^{-\mu^2/(2\sigma^2)}(1-G(s)) \bk{
	\frac{\sinh\pr{\frac{\mu}{\sigma} \frac{u}{\sigma}}}{\mu/\sigma}
} + 2\underbrace{\bar\Phi\pr{\frac{u}{\sigma} - \frac{|\mu|}{\sigma}}}_{\le \bar\Phi
\pr{u/\sigma - 1} \le \bar\Phi\pr{\frac{u}{2 \sigma}}} \tag{$\bar\Phi$ is the
complementary Gaussian CDF}\\
&\le C_V(1-G(s)) \frac{\sinh\pr{\frac{\mu}{\sigma} \frac{u}{\sigma}}}{\mu/\sigma} +
2\bar\Phi\pr{\frac{u}{2 \sigma}} \numberthis \label{eq:intermediate_large_sigma}
\end{align*}
where the second inequality follows from directly integrating the upper bound 
\eqref{eq:survival_bound} within $|x| \le u$. Now, observe that, for $|c| < 1$ and $t >
0$, \begin{align*}
t \frac{\sinh(ct)}{ct} &\le t \frac{\sinh(|c|t)}{|c|t} \tag{$\sinh(x)/x$ is an even
function}  \\
&\le t \frac{\sinh(t)}{t} \tag{$\sinh(x)/x$ is an increasing function on $x > 0$} \\
&\le \frac{1}{2}e^{t} \tag{$\sinh(x) = (e^x-e^{-x})/2\le \frac{1}{2}e^x$}.
\end{align*}
Therefore, \begin{align*}
\eqref{eq:intermediate_large_sigma} &\le C_V (1-G(s)) \exp\pr{\frac{C}{\sqrt{\log s}} \log
s} + 2 \bar\Phi
(C \sqrt{\log s}) \\
&\le C_V (1-G(s)) \exp\pr{\frac{C}{\sqrt{\log s}} \log
s} + \exp\pr{-\frac{C^2}{2} \log s} \tag{\cref{lemma:millsratio}}.
\end{align*}
Choose $C = k$ and $ K = \exp\pr{1 \vee \frac{(2C)^2}{(k-2)^2}}$. This yields that, for
$s > K$, \[
\frac{C}{\sqrt{\log s}}  \le \frac{k-2}{2} \quad \frac{C^2}{2} = \frac{k^2}{2}.
\]
Hence, integrating with respect to $s$: \begin{align*}
&\int_K^\infty s \int_{-\infty}^\infty P_{G_1}(\theta > s \mid X=x) f_X
(x\mid \mu)\,dx\,ds \\
&\le \int_K^\infty s 
\pr{C_V (1-G(s)) \exp\pr{ \frac{k-2}{2} \log
s} + \exp\pr{-\frac{k^2}{2} \log s}} \,ds \\
&\le C_{V} C_{G_1} \int_K^\infty s^{1-k+\frac{k-2}{2}}\,ds + \int_{K}^\infty s^{-k^2/2 +1}
\,ds \\
&\le C_V C_{G_1} C_k + C_k. \tag{$1-k+(k-2)/2 < -1$ and $-k^2/2+1 < -1$}
\end{align*}
as both integrals converge and depend only on $k > 2$. Returning to 
\eqref{eq:large_sigma_bound}, this shows that \eqref{eq:large_sigma} is uniformly bounded
with a constant that depends only on $V, C_{G_1}, k$. This concludes the proof.
\end{proof}

\begin{lemma}
\label{lemma:jensen_denom}
Suppose $G_1$ has mean zero and variance bounded by $V$. Let \[
f_{G_1, \sigma}(x) \equiv \int f_X(x\mid \theta) \, G_1(d\theta).
\] Then, \[
f_{G_1,\sigma}(x) \ge \frac{1}{\sqrt{2\pi}\sigma} \exp\pr{-\frac{x^2 + V}
{2\sigma^2}}
\text{, or }
\frac{1}{f_{G_1, \sigma}(x)} \le \sqrt{2\pi} \sigma \exp\pr{\frac{x^2 + V}{2\sigma^2}}.
\numberthis \label{eq:jensen_denom}
\]
\end{lemma}

\begin{proof}
Observe that, by Jensen's inequality, \[
f_{G_1, \sigma}(x) \equiv \int f_X(x\mid \theta) \, G_1(d\theta) \ge \exp \int \log f_X
(x\mid
\theta) \,G_1(d\theta).
\]
We compute \[
\log f_X(x \mid \theta) = \log\frac{1}{\sqrt{2\pi} \sigma} - \frac{1}{2\sigma^2} 
(x-\theta)^2. 
\]
Note that $\E_{\theta \sim G_1}[(x-\theta)^2] = x^2 - 2x \E_{G_1}[\theta] + \E_
{G_1}\theta^2 \le x^2 + V$.
Thus \eqref{eq:jensen_denom} follows. \end{proof}

\begin{lemma}
\label{lemma:millsratio}
For all $x \ge 0$, $\bar\Phi(x) \le \frac{1}{2} e^{-x^2/2}$. 
\end{lemma}
\begin{proof}
Note that $\bar\Phi(0) = \frac{1}{2}$ and thus the bound holds with equality at $x=0$.
Differentiate, \[
\bar \Phi'(x) = -\varphi(x) \quad \frac{d}{dx} \frac{1}{2} e^{-x^2/2} = -\frac{x}{2} e^
{-x^2/2} = (-x \sqrt{\pi/2}) \varphi(x)
\]
For $x \in [0, \sqrt{2/\pi}]$, \[
\frac{d}{dx}\bar\Phi(x) \le \frac{d}{dx} \frac{1}{2} e^{-x^2/2} \implies \bar\Phi(x) \le 
\frac{1}{2} e^{-x^2/2}. 
\]
Note that since Mill's ratio is bounded by $1/x$, we have that for all $x > 0$ \[
\bar\Phi(x) \le \varphi(x)/x.
\]
Take $x > \sqrt{2/\pi}$, we have that\[
\bar \Phi(x) \le \varphi(x) \sqrt{\frac{\pi}{2}} = \frac{1}{2} e^{-x^2/2}.
\]
Hence the inequality holds for all $ x \ge 0 $.
\end{proof}

\bibliographystyle{ecca}
\bibliography{../main.bib}

\end{document}